\documentclass[11pt]{article}

\usepackage{epsfig,amsmath,latexsym,amssymb}
\usepackage{graphicx}
\usepackage{lscape}
\usepackage{picture, eso-pic, tikz} 
\usepackage{dsfont}
\usepackage{listings}
\usepackage{changes}
\usepackage{hyperref}
\usepackage{bbding}
\usepackage{tikz-cd}
\usepackage{mathtools}

\renewcommand{\baselinestretch}{1.1}
\def\R{{\mathbb R}}  
\def\p{{\mathbb P}}  
\def\E{{\mathbb E}}  
\def\V{{\mathbb V}}  
\def\Beweis{\footnotesize}
\DeclareMathOperator*{\argmin}{arg\,min}

\newcommand{\Remm}[1]{}
\newtheorem{theo}{Theorem}[section]

\newtheorem{cor}[theo]{Corollary}
\newtheorem{defi}[theo]{Definition}
\newtheorem{model ass}[theo]{Model Assumptions}

\def\EndProof{\hfill {\scriptsize $\Box$}}

\numberwithin{equation}{section}

\definecolor{MyGray}{rgb}{0.92,0.92,0.92}


\definecolor{British racing}{rgb}{0.0, 0.5, 0.0}

\def\be{\boldsymbol{e}}

\def\bX{\boldsymbol{X}}

\def\bm{\boldsymbol{m}}

\def\bbeta{\boldsymbol{\beta}}

\def\b0{\boldsymbol{0}}

\def\bu{\boldsymbol{u}}

\newcommand{\tr}{\operatorname{trace}}

\usepackage{todonotes}
\newcommand{\Comments}{1}
\newcommand{\mynote}[2]{\ifnum\Comments=1\textcolor{#1}{#2}\fi}
\newcommand{\mytodo}[2]{\ifnum\Comments=1%
  \todo[linecolor=#1!80!black,backgroundcolor=#1,bordercolor=#1!80!black]{#2}\fi}

\ifnum\Comments=1               
\fi

\ifnum\Comments=1               
\fi

\ifnum\Comments=1               
\fi

\begin{document}
\author{
Mario V.~W{\"u}thrich\footnote{\normalfont Department of Mathematics, ETH Zurich; mario.wuethrich@math.ethz.ch}}

\date{\today}
\title{{\sc The Balance Property: The Constrained Case,\\ with a View on Risk Sharing}}
\maketitle

\begin{abstract}
\noindent  
The balance property is an important property of fitted statistical models deployed for insurance pricing. It guarantees that the total actuarial price in the fitted model is equal to the totally observed loss used to fit the model. This can be seen as an in-sample global unbiasedness property. Maximum likelihood fitted generalized linear models (GLMs) with canonical links automatically fulfill the balance property. Lindholm--W\"uthrich (Scandinavian Actuarial Journal, 2026) discussed two popular balance correction methods in case the balance property fails to hold. This note extends this discussion with a third method,  constrained GLM fitting, that turns out to be superior over the two previously discussed ones. Moreover, we highlight the connection between  the balance property and ex-post risk sharing rules.

\bigskip

\end{abstract}

\section{Introduction}
\label{sec: Introduction}
B\"uhlmann--Gisler \cite{BG} recognized the importance of the {\it balance property} in insurance pricing. The balance property ensures that the global price level is correctly specified. 
B\"uhlmann--Gisler \cite[Theorem 4.5]{BG} proved that the homogeneous B\"uhlmann--Straub \cite{BS} credibility estimator fulfills the balance property. For maximum likelihood estimated (MLE) generalized linear models (GLMs) with canonical links this property was verified to hold by Nelder--Wedderburn \cite[formula (10)]{Nelder}. These results have also found their way into neural network fitting to rectify the balance property; see W\"uthrich \cite{Wbalance}. However, these GLM and network balance property results only hold under the canonical link choice.

Under non-canonical link GLMs, a balance correction step is necessary to rectify the balance property. Lindholm--W\"uthrich \cite{LW} studied the two most popular correction methods in industry practice, a simple post-processing (SPP) method that simply shifts the intercept parameter of the GLM correspondingly, and a quasi maximum likelihood estimation (QMLE) method that fits the model under a ``wrong'' distributional assumption enforcing the selected link to be the canonical one in the ``wrong'' model.
The conclusions of Lindholm--W\"uthrich \cite{LW} were that there is no general preference for one of the two methods in terms of model accuracy, but the choice depends on the specific modeling problem to be solved.

In Lindholm--W\"uthrich \cite[Section 3.3]{LW} there was a third proposal being based on constrained MLE. This third method was not further pursued because, first, it is not implemented in standard software tools and, second, it is slightly more demanding analytically. This omission was an oversight, and we should have solved the constrained MLE in the first place. Intuitively, the constrained case should be superior to the SPP method, because it selects the optimal solution in the identical parameter space as the SPP method, but the SPP method performs this in a decoupled two step fitting procedure that does not jointly optimize over both steps, thus, it should be worse than the constrained MLE case. Intuitively, the constrained case should also outperform the QMLE method, as it solves the same problem but under the correct data generating model -- solving the same problem needs to be interpreted in the sense that the regularization term is not binding in the QMLE because the side constraint is automatically satisfied, and working in the correct data generating model is superior as proved in Gourieroux et al.~\cite{GourierouxMontfortTrogon}.

This paper proves that indeed the constrained MLE case is superior over the SPP and the QMLE methods. Henceforth, it should generally be preferred in practice; see Theorems \ref{theorem SC1} and \ref{theo main result}. These two theorems were proved with ChatGPT, and since these results are of high practical relevance in industry, we decided to make them available through this non-peer-reviewed manuscript.

Our final section highlights the close connection between the constrained MLE problem and risk sharing pools. Risk sharing pools may be viewed as a form of peer-to-peer insurance in which the aggregate loss is allocated ex-post among the pool participants. From this perspective, it can be interpreted as a total claim allocation mechanism; see Denuit--Dhaene \cite{DD}. Naturally, such a sharing mechanism must satisfy the aforementioned balance property.

\medskip

{\bf Disclaimer.}
The proofs of  Theorems  \ref{theorem SC1} and \ref{theo main result} were generated by ChatGPT and subsequently reviewed and refined by the author. No GenAI tools were used in the writing of the remainder of this manuscript.

\section{The balance property}
\label{The balance property within generalized linear models}

Consider the tuple $(Y,\bX,W) \sim \p$, with $Y$ being a real-valued response (insurance claim), $\bX$ describing the $(q+1)$-dimensional real-valued covariates, and $W$ being a strictly positive case weight. $\p$ corresponds to the portfolio distribution of a given insurer.
Actuarial pricing of response $Y$, given $(\bX, W)$, is based on the regression function
\begin{equation}\label{def: best-estimate 0}
(\bX, W)~\mapsto ~ \mu(\bX) = \E \left[\left. Y\right|\bX, W \right] \qquad \text{ a.s.}
\end{equation}
This mapping \eqref{def: best-estimate 0} involves the implicit assumption that the conditionally expected response $Y$ does not depend on $W$, given $\bX$, meaning that $Y$ considers the weight-scaled quantity. This is a common model assumption that is fulfilled, e.g., within the exponential dispersion family (EDF); see W\"uthrich--Merz \cite{WM2023}. This assumption also implies that we can drop the conditioning on $W$ in \eqref{def: best-estimate 0} because the actuarial price \eqref{def: best-estimate 0} for $Y$ does not have a sensitivity in $W$.

Typically, the true regression function $\mu(\cdot)$ is unknown and needs to be estimated from an i.i.d.~learning sample ${\cal D}=(Y_i,\bX_i, W_i)_{i=1}^n$ that follows the same law as $(Y,\bX, W)$. This procedure provides us with an estimated regression function $\widehat{\mu}(\cdot)$ that serves as an approximation to the actuarial price \eqref{def: best-estimate 0}. 
Note that the estimated regression function $\widehat{\mu}(\cdot)=\widehat{\mu}_{\cal D}(\cdot)$ generally depends on the realization of the learning sample ${\cal D}$.

\begin{defi}[balance property]  \label{definition balance property}
The balance property holds for a regression fitting procedure
if for $\p$-a.e.~realization of the learning sample ${\cal D}=(Y_i,\bX_i, W_i)_{i=1}^n$, the resulting estimated regression function $\widehat{\mu}(\cdot)=\widehat{\mu}_{\cal D}(\cdot)$ fulfills 
\begin{equation}\label{balance property formula}
\sum_{i=1}^n W_i \,\widehat{\mu}(\bX_i) = 
\sum_{i=1}^n W_i\,Y_i.
\end{equation}
\end{defi}
This is an in-sample property that is straightforward to verify, and which essentially says that we have (in-sample) global unbiasedness.

\section{Generalized linear models and canonical links}
\label{Generalized linear models}
For ease of notation, we assume throughout the document that the initial component of $\bX$ is equal to one, i.e., the covariates take the values
\begin{equation*}
  \bX=(X_0,X_1,\ldots, X_q)^\top=(1,X_1,\ldots, X_q)^\top~\in~ \{1\} \times \R^q.
\end{equation*}
Assuming that the true regression function $\mu(\cdot)$ follows a GLM, we set the GLM model equation
\begin{equation}\label{GLM assumption}
\bX ~\mapsto ~ g(\mu(\bX;\bbeta^+))= \beta^+_0 + \sum_{j=1}^q 
\beta^+_j X_j =: \left\langle \bbeta^+ , \bX \right\rangle,
\end{equation}
with strictly increasing and smooth link function $g(\cdot)$ and true regression parameter $\bbeta^+=(\beta^+_0, \ldots, \beta^+_q)^\top\in \R^{q+1}$.

If we furthermore assume that the response $Y$ follows a member of the EDF with cumulant function $\kappa$, given $(\bX, W)$, providing us with a conditional mean having GLM structure \eqref{GLM assumption}, then we can estimate the regression parameter $\bbeta \in \R^{q+1}$ with MLE; we refer to W\"uthrich--Merz \cite[Chapter 5]{WM2023} for technical details and the following lines give a short summary.\footnote{We generally assume that the selected member of the EDF has a steep cumulant function $\kappa$. This aligns the mean parameter space with the support of the responses. The members of the EDF used in practice have steep cumulant functions; see W\"uthrich--Merz \cite[Section 2.2.4]{WM2023}.}

The {\it canonical link} of the selected EDF with cumulant function $\kappa$ is given by $h=(\kappa')^{-1}$. This allows us to compute the deviance loss function 
\begin{equation}\label{deviance loss}
L_\kappa(y,m) = y h(y) - \kappa\left(h(y)\right)
-y  h(m)
+\kappa\left(h(m)\right).
\end{equation}
The MLE of regression parameter $\bbeta$ for given learning sample ${\cal D}=(Y_i,\bX_i, W_i)_{i=1}^n$
is found by solving
\begin{equation}\label{MLE 2}
\widehat{\bbeta} 
\,\in \, \argmin_{\bbeta \in \R^{q+1}}\,
\frac{1}{n} \sum_{i=1}^n \frac{W_i}{\varphi}\,L_\kappa\left(Y_i,\mu(\bX_i;\bbeta)\right).
\end{equation}
This provides us with the MLE fitted regression function
\begin{equation}\label{def: best-estimate estimate}
\bX ~\mapsto ~ \mu(\bX;\widehat{\bbeta})=g^{-1} \left\langle \widehat{\bbeta} , \bX \right\rangle.
\end{equation}

The following result has been given in Nelder--Wedderburn \cite[formula (10)]{Nelder}.
\begin{theo}\label{canonical link}
  Consider the MLE fitted GLM regression function \eqref{def: best-estimate estimate} under the canonical link choice $g=h=(\kappa')^{-1}$ within the chosen EDF with steep cumulant function $\kappa$. The MLE regression fitting procedure \eqref{MLE 2} fulfills the balance property \eqref{balance property formula}.
\end{theo}

Thus, under the canonical link choice $g=h$, MLE within GLMs fulfills the balance property, and otherwise generally not. In the general link case $g \neq h$, the MLE score equations read as
\begin{eqnarray}\nonumber
\Psi(\bbeta)&:=&
\nabla_{\bbeta}\left[
\frac{1}{n} \sum_{i=1}^n \frac{W_i}{\varphi}\, L_\kappa\left(Y_i,\mu(\bX_i;\bbeta)\right)\right]
  \\&=&\label{MLE general}
        \frac{1}{n} \sum_{i=1}^n \frac{W_i}{\varphi}\,
\frac{\mu(\bX_i;\bbeta)-Y_i}{V(\mu(\bX_i;\bbeta))}\,
\frac{1}{g'(\mu(\bX_i;\bbeta))}\, \bX_i~=~\b0,
\end{eqnarray}
with variance function $m\mapsto V(m)=\kappa''(h(m))$.
The MLE $\widehat{\bbeta}$ of $\bbeta$ is found by analyzing all solutions of \eqref{MLE general} and, in general, the balance property will fail to hold for this MLE solution under a general link choice $g\neq h$. This was precisely the motivation of Lindholm--W\"uthrich \cite{LW} to present balance correction methods to rectify this balance property.

\section{Rectifying the balance property}
Under non-canonical links $g\neq h$, the balance property for MLE fitted GLMs generally fails to hold, and Lindholm--W\"uthrich \cite{LW} proposed the following balance correction methods:
\begin{itemize}
\item Quasi  maximum likelihood estimation (QMLE) method
\item Simple post-processing (SPP) method -- probably the industry standard
\item Optimization under side constraint (SC) method
\end{itemize}
On purpose, we change the order of the methods compared to Lindholm--W\"uthrich \cite{LW}, because this will be more convenient didactically.

\medskip

In the remainder of this document, we are going to work under the technical assumptions given in Gourieroux et al.~\cite[Appendix 1]{GourierouxMontfortTrogon}, and we also refer to the related literature
Gallant--Holly \cite{GallantHolly},
Burguete et al.~\cite{Burguete} and White \cite{White} on asymptotic behavior of M- and Z-estimators. An estimate $\widehat{\bbeta}$ is based on an i.i.d.~learning sample ${\cal D}=(Y_i,\bX_i, W_i)_{i=1}^n$ of sample size $n$. Whenever the sample size is important, e.g., for asymptotic properties, we add a lower index $\widehat{\bbeta}=\widehat{\bbeta}_n$.

\subsection{Quasi maximum likelihood estimation}
\label{sec: Quasi log-likelihood}
The selected EDF distribution with steep cumulant function $\kappa$ gives us the deviance loss $L_\kappa$, see \eqref{deviance loss}. Moreover, assume a GLM structure  \eqref{GLM assumption} for the mean with a smooth and strictly increasing link function $g\neq h$. This gives us the score equations \eqref{MLE general}, and the balance property fails to hold because the link choice is not aligned with the cumulant function $\kappa$ of the selected EDF of $Y$, given $(\bX,W)$.

A crucial observation at this stage is that the data generating mechanism of $(Y,\bX,W)$ is completely irrelevant in Theorem \ref{canonical link}, but this theorem is solely a statement about the estimation procedure. This estimation procedure is based on MLE within the EDF, and this can be applied to {\it any} learning sample ${\cal D}$, regardless of how it was generated  -- this is precisely the part that is tweaked to rectify the balance property with QMLE. Assume that the responses $Y$, given $(\bX,W)$, satisfy the GLM assumption \eqref{GLM assumption} and the following conditional variance property, respectively,
\begin{equation}\label{first moment assumptions}
  \E \left[\left. Y \right| \bX, W  \right] = 
  \mu(\bX;\bbeta^+)=
  g^{-1} \left\langle \bbeta^+, \bX \right\rangle
  \qquad \text{ and } \qquad 
  \V \left(\left. Y \right| \bX, W  \right) = \frac{\varphi}{W} \,\sigma^2(\bX),
\end{equation}
for a given strictly positive function $\sigma^2(\cdot)$. These properties \eqref{first moment assumptions} do not require that we have an EDF distributed response $Y$, given $(\bX,W)$.

\medskip

Assume that the selected link $g$ allows for an EDF distribution choice with steep cumulant function $\kappa_g$ such that its canonical link satisfies $g=(\kappa'_g)^{-1}$. Solve the QMLE problem
\begin{equation}\label{in-processing step}
\widehat{\bbeta}^{\rm QMLE} \,\in \, \argmin_{\bbeta}\,
\frac{1}{n} \sum_{i=1}^n \frac{W_i}{\varphi}\,L_{\kappa_g}\left(Y_i,\mu(\bX_i;\bbeta)\right).
\end{equation}
An easy consequence of Theorem \ref{canonical link} is the following corollary.
\begin{cor}\label{canonical link cor}
The QMLE fitted model \eqref{in-processing step} fulfills the balance property \eqref{balance property formula}.
\end{cor}  

Moreover, there is the following theorem.

\begin{theo}[Gourieroux et al.~{\cite[Theorem 1]{GourierouxMontfortTrogon}}]
  \label{theorem 1}
Assume \eqref{first moment assumptions}. 
  The QMLE is strongly consistent, that is, $\widehat{\bbeta}^{\rm QMLE}_n$ converges to $\bbeta^+$,
$\p$-a.s., for $n\to \infty$.
\end{theo}

Next, we determine the speed of convergence of this QMLE. Assume that the conditional response fulfills \eqref{first moment assumptions} with a general link function $g$. Moreover, select an EDF distribution with cumulant function $\kappa$ providing the canonical link $h$ and the variance function $m\mapsto V(m)=\kappa''(h(m))$. We introduce the following matrices (assuming they exist and have full rank)
\begin{eqnarray}
{\cal J}&=&\E \left[
    \frac{W}{\varphi}\, \frac{1}{V(\mu(\bX;\bbeta^+))}\,
\left(\frac{1}{g'(\mu(\bX;\bbeta^+))}\right)^2 \bX\bX^\top
     \right],   \label{definition J}
 \\
 {\cal I}  &=&\E \left[
    \frac{W}{\varphi} \frac{\sigma^2(\bX)}{V(\mu(\bX;\bbeta^+))^2}\,
\left(\frac{1}{g'(\mu(\bX;\bbeta^+))}\right)^2 \bX\bX^\top
     \right],
     \label{definition I}
\\\label{definition H}
{\cal H}&=&
\E \left[
    \frac{W}{\varphi}\, \frac{1}{\sigma^2(\bX)}\,
\left(\frac{1}{g'(\mu(\bX;\bbeta^+))}\right)^2 \bX\bX^\top
     \right].            
\end{eqnarray}
We have the following asymptotic normality result for the MLE \eqref{MLE 2}.

\begin{theo}[Gourieroux et al.~{\cite[Theorem 3 and Property 5]{GourierouxMontfortTrogon}}]
  \label{theorem 2} Assume \eqref{first moment assumptions}. 
  The MLE \eqref{MLE 2} is asymptotically normal with
  \begin{equation*}
    \sqrt{n} \left(\widehat{\bbeta}_n- \bbeta^+\right)~ \Longrightarrow~ {\cal N}\left(\b0, \Sigma \right) \qquad \text{ for $n\to \infty$,}
  \end{equation*}
  with covariance matrix $\Sigma={\cal J}^{-1}{\cal I}
      {\cal J}^{-1}$.
We have the lower bound
\begin{equation}\label{asymptotic best}
\Sigma={\cal J}^{-1} {\cal I} {\cal J}^{-1}
~\ge~ 
{\cal H}^{-1},
   \end{equation}
the inequality refers to positive semi-definiteness.
\end{theo}

The main finding of Gourieroux et al.~\cite{GourierouxMontfortTrogon} is that the lower bound in Theorem \ref{theorem 2} is only attained if we work under the correct variance assumption in the estimation procedure
\begin{equation}\label{variance optimal}
\sigma^2(\bX)=
V(\mu(\bX;\bbeta^+))
=\kappa'' \left(h(\mu(\bX;\bbeta^+))\right).
\end{equation}
Thus, we need to select the cumulant function $\kappa$ such that \eqref{variance optimal} holds. This aligns the deviance loss $L_\kappa$ with the properties of the response, making the estimation procedure so-called {\it best asymptotically normal}. If at the same time the balance property should hold, this fully determines the possible link choice $g=h=(\kappa')^{-1}$. In any other variance case, the QMLE \eqref{in-processing step} will not be best asymptotically normal, i.e., will have a non-minimal asymptotic covariance matrix in the sense of \eqref{asymptotic best}.

\medskip

Applying these results to the QMLE works as follows. For the given link in \eqref{first moment assumptions}, select the cumulant function $\kappa_g$ with canonical link, i.e., $g=(\kappa_g')^{-1}$. This gives us the variance function  $m\mapsto V_g(m)=\kappa_g''(g(m))$. Inserting this variance function into \eqref{definition J} and \eqref{definition I} gives us ${\cal J}_g$ and ${\cal I}_g$, respectively. Using Theorem \ref{theorem 2} gives us the asymptotic normality
  \begin{equation}\label{asymptotic QMLE}
    \sqrt{n} \left(\widehat{\bbeta}^{\rm QMLE}_n- \bbeta^+\right)~ \Longrightarrow~ {\cal N}\left(\b0, \Sigma_{\rm QMLE} \right) \qquad \text{ for $n\to \infty$,}
  \end{equation}
  with asymptotic variance
\begin{equation*}
\Sigma_{\rm QMLE}={\cal J}_g^{-1} {\cal I}_g {\cal J}_g^{-1}~\ge~ {\cal H}^{-1}.
   \end{equation*}  
Best asymptotically normal is achieved if the canonical link $g$ of $\kappa_g$ provides the variance behavior \eqref{variance optimal} in terms of the true conditional variance of $Y$, given $(\bX, W)$, and in any other case we receive a bigger asymptotic covariance matrix -- though in any case the balance property is fulfilled in this QMLE approach \eqref{in-processing step}.

\subsection{Simple post-processing}
\label{Post-processing}
Starting from \eqref{MLE 2}, we apply a post-processing step by shifting the intercept estimate $\widehat{\beta}_0$ 
of the MLE $\widehat{\bbeta}$ by the (unique) constant $\gamma \in \R$ solving the identity
\begin{equation}\label{intercept shift}
\sum_{i=1}^n W_i \,\mu(\bX_i;\widehat{\bbeta}+\gamma \, \be_1) = 
\sum_{i=1}^n W_i \,g^{-1}\left\langle \widehat{\bbeta}+\gamma \, \be_1, \bX_i
\right\rangle=
\sum_{i=1}^n W_i\,Y_i,
\end{equation}
with first basis vector $\be_1=(1,0,\ldots, 0)^\top
\in \R^{q+1}$.
Denote the solution of \eqref{intercept shift} by $\widehat{\gamma}$. This
gives the SPP estimate
\begin{equation*}
\widehat{\bbeta}^{\rm SPP} =(\widehat{\beta}_0 + \widehat{\gamma }, 
\widehat{\beta}_1,\ldots, \widehat{\beta}_q)^\top~\in~\R^{q+1}.
\end{equation*}
We study the asymptotic behavior of this balance corrected estimator as the sample size $n$ goes to infinity. We introduce 
\begin{equation}\label{definition vector a}
\bm=
\E \left[\frac{W}{\varphi}\, \frac{1}{g'(\mu(\bX;\bbeta^+))}\, \bX\right]
\in \R^{q+1},
\qquad \tau^2=\E\left[\frac{W}{\varphi}\sigma^2(\bX)\right],
\qquad
 a=\bm^\top {\cal H}^{-1}\bm.
\end{equation}
Since ${\cal H}$ is positive definite and $\bm\neq 0$, one has $a>0$; note that the initial component of $\bm=(m_0,m_1,\ldots, m_q)^\top$ satisfies the strict positivity $m_0 = \E[W/(\varphi\, g'(\mu(\bX;\bbeta^+)))]>0$. There is the following result.

\begin{theo}[Lindholm--W\"uthrich {\cite[Theorems 4.6-4.7]{LW}}] \label{corollary asymptotic simple}
Based on the correct variance choice \eqref{variance optimal} of the cumulant function $\kappa$ and any link $g$
we have strong consistency and asymptotic normality of the SPP estimate
  \begin{equation}\label{SPP asymptotic N}
    \sqrt{n} \left(\widehat{\bbeta}_n^{\rm SPP}- \bbeta^+\right)~ \Longrightarrow~ {\cal N}\left(\b0, \Sigma_{\rm SPP} \right) \qquad \text{ for $n\to \infty$,}
  \end{equation}
with covariance matrix
\begin{equation*}
\Sigma_{\rm SPP}~=~{\cal H}^{-1}+
\frac{\tau^2 -a}{m^2_0}\, \be_1 \be_1^\top,
\end{equation*}
with first basis vector $\be_1=(1,0,\ldots, 0)^\top \in \R^{q+1}$ 
and with $\tau^2\ge a>0$.
\end{theo}

\subsection{Optimization under side constraints}
Assume we work under the correct variance choice \eqref{variance optimal} for the selected steep cumulant function $\kappa$ and a general link $g$. We introduce the balance side constraint 
\begin{equation}\label{Gamma}
{\cal C}(\bbeta):=\frac{1}{n}\sum_{i=1}^n \frac{W_i}{\varphi}\left(\mu(\bX_i;\bbeta)-Y_i\right)=0.
\end{equation}
In view of the MLE problem \eqref{MLE 2}, we consider in the constrained case
\begin{equation}\label{MLE 2 constraint}
\widehat{\bbeta}^{\rm sc} 
\,\in \, \argmin_{\bbeta:\, {\cal C}(\bbeta)=0}~
\frac{1}{n} \sum_{i=1}^n \frac{W_i}{\varphi}\,L_\kappa\left(Y_i,\mu(\bX_i;\bbeta)\right).
\end{equation}
We assume that the constrained problem has a local solution in a neighborhood of $\bbeta^+$ and that the usual Karush--Kuhn--Tucker (KKT) conditions hold there. Recall the unconstrained score $\Psi(\bbeta)$ from \eqref{MLE general}, and set for the side constraint
\begin{equation*}
\Gamma(\bbeta)=\nabla_{\bbeta}\, {\cal C}(\bbeta)
=\frac{1}{n}\,\sum_{i=1}^n \frac{W_i}{\varphi\,g'(\mu(\bX_i;\bbeta))}\bX_i.
\end{equation*}
The Lagrange function gives the KKT equations in $(\bbeta, \lambda)$
\begin{equation}
  \begin{pmatrix}
\Psi(\bbeta)+\lambda \Gamma(\bbeta) \\
    {\cal C}(\bbeta)
    \end{pmatrix} =\b0,
\label{eq:KKT}
\end{equation}
with Lagrange multiplier $\lambda\in\R$.

\begin{theo}
\label{theorem SC1}
  Under the correct variance choice \eqref{variance optimal} and any link function $g$
we have asymptotic normality of the constrained estimate
\begin{equation}\label{asymptotics SC1}
\sqrt{n}\left(\widehat{\bbeta}_n^{\rm sc}-\bbeta^+\right)~
\Longrightarrow ~{\cal N}\left(\b0,\Sigma_{\rm sc}\right)\qquad \text{ for $n\to \infty$,}
\end{equation}
with covariance matrix
\begin{equation*}
  \Sigma_{\rm sc}
  = {\cal H}^{-1} + \frac{\tau^2 - a}{a^2}\, {\cal H}^{-1}\bm \bm^\top {\cal H}^{-1},
\end{equation*}
and with $\tau^2 \ge a >0$.
\end{theo}

Implementation of this constrained estimation method \eqref{MLE 2 constraint} requires to compute the Hessian in $(\bbeta, \lambda)$ which can be obtained from \eqref{eq:KKT}. This then allows one to find the constrained MLE from a learning sample by using quasi-Newton methods. For fast convergence, one can initialize this algorithm in the unconstrained MLE $\widehat{\bbeta}$.

\medskip

{\Beweis
{\bf Proof of Theorem \ref{theorem SC1}.}
Under the correct variance assumption \eqref{variance optimal}, we have asymptotic normality
\begin{equation}\label{asymtptotic normality 17}
  \sqrt{n}  \begin{pmatrix}
    \Psi(\bbeta^+)\\ {\cal C}(\bbeta^+)
    \end{pmatrix}
~\Longrightarrow~
{\cal N}\left(\b0,
\begin{pmatrix}
{\cal H} & \bm\\
\bm^\top & \tau^2
\end{pmatrix}
\right) \qquad \text{ for $n\to \infty$.}
\end{equation}
Next, we consider a first-order Taylor expansion of the left-hand side of 
\eqref{eq:KKT} around the unconstrained solution $(\bbeta^+,\lambda^+=0)$. This gives us
\begin{equation*}
\sqrt{n}  \begin{pmatrix}
 \Psi(\bbeta^+)+ {\cal H}\left(\bbeta-\bbeta^+\right)+\bm \lambda
\\
 {\cal C}(\bbeta^+) +  \bm^\top \left(\bbeta-\bbeta^+\right) 
    \end{pmatrix} + o_p(1) =\b0,
\end{equation*}
where $o_p(1)$ is vanishing in probability as $n\to \infty$. This can be rewritten as
\begin{eqnarray*}
\sqrt{n}  
 \begin{pmatrix}
\bbeta-\bbeta^+  
\\
 \lambda
\end{pmatrix}
&=&
- \sqrt{n}\begin{pmatrix}
 {\cal H} & \bm 
\\
 \bm^\top & 0
\end{pmatrix}^{-1}  \begin{pmatrix}
 \Psi(\bbeta^+) 
\\
 {\cal C}(\bbeta^+)
\end{pmatrix}
+ o_p(1)
\\&=&    
- \sqrt{n}\begin{pmatrix}
 {\cal H}^{-1} - \frac{{\cal H}^{-1}\bm \bm^\top {\cal H}^{-1}}{a} & \frac{{\cal H}^{-1}\bm }{a} 
\\\frac{ \bm^\top {\cal H}^{-1}}{a} & -\frac{1}{a}
\end{pmatrix}  
\begin{pmatrix}
 \Psi(\bbeta^+) 
\\
 {\cal C}(\bbeta^+)
\end{pmatrix}
+ o_p(1).
\end{eqnarray*}
Focusing on the regression parameter components, we get
\begin{eqnarray*}
  \sqrt{n}\left(\bbeta-\bbeta^+\right)
  &=&
      -\left({\cal H}^{-1} - \frac{{\cal H}^{-1}\bm \bm^\top {\cal H}^{-1}}{a}\right)
      \sqrt{n}\Psi(\bbeta^+)
  + \frac{{\cal H}^{-1}\bm }{a}\,\sqrt{n}{\cal C}(\bbeta^+) + o_p(1).
\end{eqnarray*}
Next, applying the asymptotic normality \eqref{asymtptotic normality 17} in the previous identity to $\widehat{\bbeta}_n^{\rm sc}$, as $n \to \infty$, gives us the
claim with covariance matrix
\begin{equation*}
  \Sigma_{\rm sc}
  = {\cal H}^{-1} + \frac{\tau^2 - a}{a^2}\, {\cal H}^{-1}\bm \bm^\top {\cal H}^{-1}.
\end{equation*}
This completes the proof.
\EndProof}

\section{Forecast accuracy}
The main result of Lindholm--W\"uthrich \cite{LW} was that one cannot give a general preference to either the QMLE or the SPP method w.r.t.~forecast accuracy, because their ranking is problem and situation specific. To measure {\it forecast accuracy}, Lindholm--W\"uthrich \cite{LW} were starting from the average Kullback--Leibler (KL) divergence  to measure the difference from an estimated model $\mu(\bX ;\widetilde{\bbeta})$ to the true model $\mu(\bX; \bbeta^+)$ that belong to the same GLM-EDF class but differ in their parameters $\widetilde{\bbeta}$ and $\bbeta^+$.
Averaging over the population distribution $\bX \sim \p$, one can approximate this average KL divergence by, see Lindholm--W\"uthrich \cite[Lemma 4.8]{LW},
\begin{equation}\label{KL trace}
\E\left[D_{\rm KL}(\bbeta^+\| \widetilde{\bbeta})\right]
~\approx~
\frac1{2n}\tr\left({\cal H}\widetilde{\Sigma}\right),
\end{equation}
assuming that $\widetilde{\bbeta}=\widetilde{\bbeta}_n$ is  
asymptotically normal for $\bbeta^+$ with rate $\sqrt{n}$ and asymptotic covariance matrix $\widetilde{\Sigma}$ for $n \to \infty$.

Approximation  \eqref{KL trace} gives the leading term of the asymptotic behavior as $n \to \infty$, the rate is  given by $(2n)^{-1}$ and the leading constant is determined by $\tr({\cal H}\widetilde{\Sigma})$. Measuring forecast accuracy of the different presented model estimation methods, we thus compute the traces of the different asymptotic covariance matrices, and the smaller the trace the better the (asymptotic) approximation to the true model. This gives a precise meaning to what we coin {\it superior} in the introduction Section \ref{sec: Introduction}.

In Lindholm--W\"uthrich \cite{LW} it was shown by examples that one cannot generally rank the traces $\tr({\cal H}{\Sigma}_{\rm QMLE})$ of the QMLE approach \eqref{asymptotic QMLE} and $\tr({\cal H}{\Sigma}_{\rm SPP})$ of the SPP approach \eqref{SPP asymptotic N}. Thus, we cannot rank these two methods w.r.t.~\eqref{KL trace}.
However, the next theorem proves that the constrained case outperforms both of them in terms of the resulting trace.

\begin{theo}\label{theo main result}
Consider the asymptotic covariance matrices of the  QMLE approach \eqref{asymptotic QMLE}, the SPP approach \eqref{SPP asymptotic N} and the constrained case \eqref{asymptotics SC1}.
We have
\begin{equation}\label{ranking obtained}
q+1
~\le~
\tr\left({\cal H}\Sigma_{\rm sc}\right)
~\le~ \min \Big\{ \tr\left({\cal H}\Sigma_{\rm QMLE}\right), \,
\tr\left({\cal H}\Sigma_{\rm SPP}\right) \Big\}.
\end{equation}
\end{theo}

\medskip

{\bf Discussion of Theorem \ref{theo main result}:}
\begin{itemize}
\item 
The lower bound $q+1$ in \eqref{ranking obtained} corresponds to the unconstrained MLE $\widehat{\bbeta}$, given in \eqref{MLE 2}, which gives the most accurate forecasts, but which does not generally satisfy the balance property. 
Theorem \ref{theo main result} then says that we should give preference to the constrained MLE $\widehat{\bbeta}^{\rm sc}$, given in 
\eqref{MLE 2 constraint}, as it outperforms the QMLE and the SPP methods.
\begin{itemize}
\item The constrained case outperforms the SPP method because the latter only corrects in the intercept direction whereas the constrained case corrects in the optimal correction.
  \item The QMLE solves the same problem as the constrained case -- the side constraint is not binding in the QMLE because the balance property is fulfilled -- however, the QMLE solves this problem in the wrong likelihood geometry w.r.t.~the variance property of the responses.
  \end{itemize}
\item There is a caveat in these results, namely, Theorems \ref{corollary asymptotic simple} and \ref{theorem SC1} require that we work under the correct variance assumption  \eqref{variance optimal}. If the correct variance assumption is unknown, one can try to infer it by an iterative estimation scheme, as proposed
  in Delong--W\"uthrich \cite{DelongW}.  
\end{itemize}

\bigskip

{\Beweis
{\bf Proof of Theorem \ref{theo main result}.}
We start by computing the trace of the constrained case
\begin{eqnarray}\nonumber
\tr({\cal H}\Sigma_{\rm sc})
&=&
\tr\left({\cal H} \left({\cal H}^{-1} + \frac{\tau^2 - a}{a^2}\, {\cal H}^{-1}\bm \bm^\top {\cal H}^{-1}\right) \right)
\\&=&\label{used later again}
\tr\left(\mathds{1} + \frac{\tau^2 - a}{a^2}\, \bm \bm^\top {\cal H}^{-1}\right) ~=~ (q+1)+\frac{\tau^2-a}{a}.
\end{eqnarray}
Next, we compute the trace of the SPP case
\begin{equation*}
\tr({\cal H}\Sigma_{\rm SPP})
=
\tr\left({\cal H} \left({\cal H}^{-1} + \frac{\tau^2 - a}{m_0^2}\, \be_1 \be_1^\top \right) \right)
=
\tr\left(\mathds{1} + \frac{\tau^2 - a}{m_0^2}\, {\cal H}_{1,1}\right)
=(q+1) + (\tau^2-a)\,\frac{{\cal H}_{1,1}}{m_0^2}.
\end{equation*}
We define the scalar product in the ${\cal H}$-metric in $\R^{q+1}$ by 
$\langle \boldsymbol{b}, \boldsymbol{a} \rangle_{\cal H} = 
\boldsymbol{b}^\top {\cal H} \boldsymbol{a}$ for $\boldsymbol{b}, \boldsymbol{a}\in \R^{q+1}$.
The Cauchy--Schwarz inequality then provides us with
\begin{eqnarray*}
m_0^2&=&(\be_1^\top \bm)^2~=~(\be_1^\top {\cal H}{\cal H}^{-1} \bm)^2
~=~\left|\langle \be_1, {\cal H}^{-1} \bm\rangle_{\cal H}\right|^2
\\&\le&
\langle \be_1, \be_1\rangle_{\cal H}\,
\langle {\cal H}^{-1} \bm, {\cal H}^{-1} \bm\rangle_{\cal H}~=~
\left(e_1^\top {\cal H}\be_1\right)\left(\bm^\top {\cal H}^{-1}\bm\right)
={\cal H}_{1,1}\,a.
\end{eqnarray*}
This proves $\tr({\cal H}\Sigma_{\rm sc}) \le \tr({\cal H}\Sigma_{\rm SPP})$, because ${\cal H}_{1,1}/m_0^2\ge 1/a$.

\medskip

Now, we turn our attention to the QMLE case. For the selected link function $g$ and the matching cumulant function $\kappa_g$, we have
\begin{equation} \label{first formulation}
{\cal J}_g=\E \left[
    \frac{W}{\varphi}\, \frac{1}{V_g(\mu(\bX;\bbeta^+))}\,
\left(\frac{1}{g'(\mu(\bX;\bbeta^+))}\right)^2 \bX\bX^\top
     \right]
     =\E \left[
    \frac{W}{\varphi}\, \frac{1}{g'(\mu(\bX;\bbeta^+))} \bX\bX^\top
     \right].
\end{equation}
From this we obtain
\begin{equation*}
     {\cal J}_g \be_1 =
\E \left[\frac{W}{\varphi}\, \frac{1}{g'(\mu(\bX;\bbeta^+))}\, \bX\right]= \bm.
\end{equation*}
This allows us to rewrite 
\begin{equation*}
a=\bm^\top {\cal H}^{-1}\bm=\be_1^\top {\cal J}_g{\cal H}^{-1}{\cal J}_g\be_1,
\end{equation*}
and applying the same computation \eqref{first formulation} to ${\cal I}_g$
\begin{equation*}
\tau^2-a= \be_1^\top {\cal I}_g \be_1 - \be_1^\top {\cal J}_g{\cal H}^{-1}{\cal J}_g\be_1= \be_1^\top \left({\cal I}_g  -  {\cal J}_g{\cal H}^{-1}{\cal J}_g\right)\be_1
= \be_1^\top{\cal B}_g\be_1
 \ge 0,
\end{equation*}
where we have set ${\cal B}_g={\cal I}_g  -  {\cal J}_g{\cal H}^{-1}{\cal J}_g$. This matrix is positive semit-definite; Gourieroux et al.~\cite[Property 5]{GourierouxMontfortTrogon}. Using \eqref{used later again} we have
\begin{equation*}
\tr({\cal H}\Sigma_{\rm sc})
= (q+1)+\frac{\tau^2-a}{a}
= (q+1)+\frac{\be_1^\top {\cal B}_g \be_1}{\be_1^\top   {\cal J}_g{\cal H}^{-1}{\cal J}_g\be_1}.
\end{equation*}
For the QMLE case we have
\begin{eqnarray*}
\tr({\cal H}\Sigma_{\rm QMLE})
&=&
\tr\left({\cal H}{\cal J}_g^{-1}{\cal I}_g{\cal J}_g^{-1}\right)
~=~
\tr\left({\cal H}{\cal J}_g^{-1}\left({\cal J}_g{\cal H}^{-1}{\cal J}_g+{\cal B}_g\right){\cal J}_g^{-1}\right)
\\&=&(q+1)+
\tr\left({\cal H}{\cal J}_g^{-1}{\cal B}_g{\cal J}_g^{-1}\right)
~=~
(q+1)+\tr\left({\cal H}^{1/2}{\cal J}_g^{-1}{\cal B}_g{\cal J}_g^{-1}{\cal H}^{1/2}\right).
\end{eqnarray*}
We introduce the positive semi-definite matrix ${\cal M}_g={\cal H}^{1/2}{\cal J}_g^{-1}{\cal B}_g{\cal J}_g^{-1}{\cal H}^{1/2}$.
We set vector  $\bu={\cal H}^{-1/2}{\cal J}_g \be_1$. This implies
\begin{equation}\label{to be used 31}
\bu^\top {\cal M}_g \bu
=
\bu^\top \left({\cal H}^{1/2}{\cal J}_g^{-1}{\cal B}_g{\cal J}_g^{-1}{\cal H}^{1/2}\right) \bu = \be_1^\top {\cal B}_g \be_1,
\end{equation}
and
\begin{equation}\label{to be used 32}
\bu^\top  \bu
=\be_1^\top {\cal J}_g{\cal H}^{-1}{\cal J}_g \be_1.
\end{equation}
The trace of  matrix ${\cal M}_g$ dominates its biggest eigenvalue giving us the inequality
\begin{equation*}
\tr({\cal M}_g)\ge \frac{\bu^\top {\cal M}_g \bu}{\bu^\top \bu}.
\end{equation*}
Using the definition of ${\cal M}_g$ and \eqref{to be used 31}-\eqref{to be used 32}, the latter is equivalent
\begin{equation*}
\tr\left({\cal H}{\cal J}_g^{-1}{\cal B}_g{\cal J}_g^{-1}\right)=
\tr({\cal M}_g)\ge \frac{\bu^\top {\cal M}_g \bu}{\bu^\top \bu}= \frac{\be_1^\top {\cal B}_g \be_1}{\be_1^\top {\cal J}_g{\cal H}^{-1}{\cal J}_g \be_1}=\frac{\tau^2-a}{a}.
\end{equation*}
This completes the proof.
\EndProof}

\section{Risk sharing}
\label{sec: Risk sharing}
We have started our discussion by emphasizing the importance of the balance property in insurance pricing to ensure in-sample global unbiasedness, being an important aspect of calibrated pricing schemes. In particular, the balance property can be verified purely data driven, and it does not require any knowledge of the true expected loss. Recall, the model fitting procedure satisfies the balance property, if for $\p$-a.e.~realization of the learning sample ${\cal D}=(Y_i,\bX_i,W_i)_{i=1}^n$, the resulting fitted model $\widehat{\mu}_{\cal D}(\cdot)$ satisfies the equality
\begin{equation}\label{balance property formula 2}
\sum_{i=1}^n W_i \,\widehat{\mu}_{\cal D}(\bX_i) = 
\sum_{i=1}^n W_i\,Y_i.
\end{equation}
This is a property of the fitted model $\widehat{\mu}_{\cal D}(\cdot)$, which is then going to be deployed for actuarial pricing of the next calendar year -- also called out-of-time in the sense of out-of-sample, in contrast to in-sample. The prices $W_t \,\widehat{\mu}_{\cal D}(\bX_t)$  are, thus, set ex-ante for the next calendar year and, generally, the future realized claims $Y_t$ will differ on the portfolio level
\begin{equation}\label{ex-ante future}
\sum_{t=1}^m W_t \,\widehat{\mu}_{\cal D}(\bX_t) \neq 
\sum_{t=1}^m W_t\,Y_t,
\end{equation}
if $(Y_t,\bX_t,W_t)_{t=1}^m$ denotes the portfolio observation of the next calendar year.
These considerations concern classical ex-ante actuarial pricing within the insurance industry, and possible shortfalls in \eqref{ex-ante future} need to be covered by the risk capital buffer.

One stream of recent actuarial research has been focusing on peer-to-peer insurance where the scheme participants share the experienced total loss ex-post in a predefined risk sharing scheme (total loss allocation mechanism); see Denuit--Dhaene \cite[Definition 3.1]{DD}. In particular, assume that $(Y_i)_{i=1}^n$ are the individual losses of the scheme participants $(\bX_i,W_i)_{i=1}^n$. According to the terms of the risk sharing scheme, find measurable functions $h_i(\cdot)$ such that
\begin{equation}\label{actual loss observed}
  \sum_{i=1}^n h_i\left((Y_i, \bX_i, W_i)_{i=1}^n\right) = \sum_{i=1}^n W_i Y_i=:S.
\end{equation}
This indicates the close connection of risk sharing to the balance property \eqref{balance property formula 2}.
Risk sharing schemes can, e.g., be found as a conditional mean risk sharing (CMRS) rule
\begin{equation*}
  g_i(S) = \E \left[\left. W_iY_i \right| S \right],
  \end{equation*}
  and a stream of literature studies its Pareto-optimality, e.g., in terms of convex orders; we refer to Denuit--Dhaene \cite{DD}, Denuit--Robert \cite{DR} and the literature therein. Moreover, Jiao et al.~\cite{Ruodu} give an axiomatic foundation to the CMRS rule, and computational aspects of the CMRS rule are studied in Blier-Wong \cite{BlierWong}, and we also refer to the additional literature therein.

Another interesting aspect is provided by Chong et al.~\cite{Chong}, who combine ex-ante capital allocation (from a risk management perspective) and ex-post risk sharing (from a peer-to-peer perspective). By randomizing the ex-ante capital allocation principle, these authors provide a new view on ex-post risk sharing. Denote by $K_i$ the capital buffer of risk $1\le i \le n$. Chong et al.~\cite[formula (5)]{Chong} study the ex-ante capital allocation rule under a given loss function $L$ (we adapt the case weight scaling to our EDF view)
\begin{equation}\label{capital allocation 0}
  (K^\ast_i)_{i=1}^n~\in ~
  \argmin_{(K_i)_{i=1}^n:\, \sum_{i=1}^n K_i = K}~ \sum_{i=1}^n W_i\, \E^{\mathbb{Q}_i} \left[L\left(Y_i, K_i/W_i\right)\right],
\end{equation}
for a predefined capital level $K$, deterministic case weights $W_i>0$, and where the probability measures $\mathbb{Q}_i$ capture the risk aversion and the preferences of the individual scheme participants. To turn this into a risk sharing rule, the authors transform \eqref{capital allocation 0} into a parametrized version. Let $(K(\theta), W_i(\theta), \mathbb{Q}_i^\theta)_{1\le i \le n}$ be a parametrized version of $(K,W_i,\mathbb{Q}_i)_{1\le i \le n}$ with parameter $\theta \in [0,1]$, and assume $\theta \mapsto K(\theta)$ is right-invertible. This gives the $\theta$-parametrized solutions
\begin{equation}\label{Runhuan}
  (K^\ast_i)_{i=1}^n(\theta)~\in~
  \argmin_{(K_i)_{i=1}^n:\, \sum_{i=1}^n K_i = K(\theta)}~ \sum_{i=1}^n W_i(\theta)\, \E^{\mathbb{Q}^\theta_i} \left[L\left(Y_i, K_i/W_i(\theta)\right)\right].
\end{equation}
That is, for every capital level $K(\theta)$, there is an optimal ex-ante capital allocation rule $(K^\ast_i)_{i=1}^n(\theta)$. Turning this into an ex-post risk sharing rule requires linking the capital level $K(\theta)$ to the actual loss $S$ defined in \eqref{actual loss observed}. This is obtained by selecting a $[0,1]$-valued random variable $\Theta$, such that $K(\Theta)=S$, $\p$-a.s. This then induces the ex-post risk sharing rule $(K^\ast_i)_{i=1}^n(\Theta)$ at the experienced aggregate loss level $S$.
For the square loss function $L$, and the weights $W_i(\theta)\equiv W_i$ and probability measures
$\mathbb{Q}_i^\theta\equiv \mathbb{Q}_i$ not depending on $\theta$, the $\Theta$-randomized version of \eqref{Runhuan}  gives exactly the quota-sharing risk sharing rule; see Chong et al.~\cite[fomula (12)]{Chong}.

It is important to emphasize that this ex-post risk sharing rule is a two-step procedure; see Chong et al.~\cite[Remark 4.1]{Chong}: First, solve the parametrized problem \eqref{Runhuan} and, second, set the parameter $\theta$ to the observed loss level $K(\theta)=S$. This provides {\it regularization}. If we would try to solve the problem in one step by replacing the side constraint in \eqref{Runhuan} by $\sum_{i=1}^n K_i = S$, we would get the trivial solution $K^\ast_i = W_i Y_i$ for any strictly consistent loss function $L$ for mean estimation. This is because regularization is missing if we were allowed to select any $\sigma( (Y_i)_{i=1}^n)$-measurable solution. The constrained MLE case \eqref{MLE 2 constraint} addresses regularization in a different way, namely, we assume that all insurance policyholders $1\le i \le n$ are characterized by covariates $\bX_i$ that account for the corresponding heterogeneity in their risks. Based on this, we select a (parametrized) class of regression models $\{\bX\mapsto\mu(\bX; \bbeta)\}_{\bbeta}$ on the covariates $\bX$, and this class of regression models restricts the solution space. Optimization problem \eqref{MLE 2 constraint} can then be seen as the empirical version of the resulting risk sharing rule.

\medskip

We mention some differences to Chong et al.~\cite{Chong} and discuss related items:
\begin{itemize}
\item Our loss function $L_\kappa$ is a deviance loss function, i.e., our objective function promotes at approximating the conditional expectations of $Y_i$ as good as possible, given $\bX_i$, subject
  to the balance side constraint \eqref{Gamma}. Whereas Chong et al.~\cite{Chong} consider the two examples of the square loss function and the absolute value loss function:
  \begin{itemize}
  \item  The first one is strictly consistent for mean estimation, in fact, it is a deviance loss function, but potentially not sharing the same properties as the responses; see Section \ref{sec: Quasi log-likelihood} on QMLE. Consequently, on finite samples one does not receive the optimal predictive model on average, see Theorem \ref{theorem 2}.
  \item The second one is strictly consistent for median estimation, thus, a pinball loss, and the resulting risk sharing rule is targeting medians.
  \end{itemize}
\item Chong et al.~\cite{Chong} allow for heterogeneity by allowing the scheme participants to have different risk aversions and preferences through the measures $\mathbb{Q}_i$, and the claims $Y_i$ can have different distributions. Our case \eqref{MLE 2 constraint} does not consider risk aversion, though it might be implemented by a weighting scheme through changes of measure via Radon--Nikod\'ym derivatives. However, these changes of measure would not result in the best possible predictive model, because the optimization would not take place under the real world measure $\p$ (generating the losses $Y_i$). On the other hand, \eqref{MLE 2 constraint} allows for heterogeneity by considering all available insurance policyholder information $\bX_i$, which constrains the solution space to the measurable functions of $(\bX_i)_{i=1}^n$. This prevents from the trivial risk sharing rule $K^*_i = W_i Y_i$, unless the responses are measurable w.r.t.~the covariates (which typically is not the case, because otherwise insurance would not be necessary as the claims are perfectly predictable knowing the covariates).
\item Following up on the previous item, if we drop the balance side constraint \eqref{Gamma} and if we optimize the non-empirical version over all measurable functions on $(\bX_i)_{i=1}^n$, assuming independence of policyholders under the population distribution $\p$, we obtain the true regression function under the expected deviance loss (this precisely uses the strict consistency for mean estimation of $L_\kappa$)
  \begin{equation*}
\mu(\bX_i) = \E \left[\left. Y_i\right|\bX_i \right].
\end{equation*}
Thus, our minimization tries to separate the systematic part $\mu(\bX_i)$ from the idiosyncratic (irreducible risk) part $\varepsilon_i=Y_i-\mu(\bX_i)$, and only the systematic part is charged subject to the balance side constraint. This balance side constraint is allocated to the individual pool participants such that the risk sharing perturbations are as close as possible to the systematic part w.r.t.~the selected expected deviance loss. This is a different view on ``fairness'' of ex-post risk sharing rules.

\item Much of the risk sharing literature does not address the fact that typically the true data generating model is unknown, i.e., there is an additional statistical modeling problem involved. Our approach solves this elegantly by considering the optimal constrained solution \eqref{MLE 2 constraint} of model fitting. This paper presents the GLM regression example, but naturally this works for any family of regression functions.

\item Following up on the previous item, if we work with a GLM under the canonical link choice, the MLE fitted model automatically fulfills the balance property, see Theorem \ref{canonical link}. Thus, in this case the ex-post risk sharing rule precisely charges the {\it estimated systematic part} -- this is a pretty strong result which supports the above fairness view! More generally, this can also be obtained within neural networks, see W\"uthrich \cite{Wbalance}, and gradient boosting machines (GBMs) if one works with canonical link functions that match the selected deviance loss. If one does not work in this canonical setting, then penalized (constrained) M-estimation can be applied.

\item This paper was studying the balance property \eqref{balance property formula} which can be seen as a global balance. It is well-known that if one considers a GLM with canonical link, then the balance property holds on every level of a categorical covariate; see Denuit et al.~\cite{DenuitCharpentierTrufin} and W\"uthrich--Merz \cite[formula (5.23)]{WM2023}. Thus, in fact, the ex-post risk sharing takes place on a (local) categorical covariate level. This also opens up the discussion about (algorithmic) fairness, multi-balance and multi-calibration; see Denuit et al.~\cite{Multi}.
  
\end{itemize}

This short section tried to give a high-level overview of some aspects on risk sharing connected to the balance property, and we hope that it stimulates more research in this direction.

\section{Conclusions}
This note revisited the balance property which is a crucial feature in actuarial pricing.
We presented three methods of balance correction in case the fitted models do not automatically fulfill this property: 1) quasi maximum likelihood estimation (QMLE) method, 2) simple post-processing (SPP) method, and 3) optimization under side constraint (SC) method. Within a generalized linear modeling (GLM) framework, using deviance losses for maximum likelihood estimation (MLE), it turns out that the SC method is superior over the other two methods in terms of prediction accuracy (where prediction accuracy was defined to be the asymptotic KL-trace in the infinite sample limit). Therefore, if the balance property is an essential feature of the fitted model, the recommendation is to use the SC method.

In the second part of this note, we have linked the SC method to ex-post risk sharing. Our proposal provides an ex-post risk sharing rule that uses the available policyholder information in an optimal way. If we work with a GLM under the canonical link choice, we can even coin this ex-post risk sharing solution 'fair' because it precisely charges the price of the estimated systematic effects, thus, it balances the estimated idiosyncratic (irreducible risk) part (being shared across the risk sharing pool).

The balance property is an in-sample property, and for out-of-time forecasting of future calendar years, further corrections may be necessary in case of non-stationarity, which is coined concept drift in actuarial pricing; see Brauer et al.~\cite{BrauerMenzel} for a discussion of concept drifts.

Finally, our notion of superiority and prediction accuracy relies on an asymptotic behavior under a known variance structure. Both, understanding a misspecification of the variance structure and finite sample results would be of great practical interest.

\bigskip

{\small 
\renewcommand{\baselinestretch}{.51}
}

\end{document}